\documentclass{amsart}
\usepackage{amsmath}
\usepackage{amssymb}
\usepackage{amsthm}
\usepackage{tikz}
\usepackage{hyperref}
\hypersetup{colorlinks=false}

\theoremstyle{plain}

\newtheorem{theorem}{Theorem}[section]
\newtheorem{lemma}[theorem]{Lemma}

\theoremstyle{definition}
\newtheorem{definition}[theorem]{Definition}
\newtheorem{remark}[theorem]{Remark}
\newtheorem{example}[theorem]{Example}

\numberwithin{equation}{section}
\newcommand\fantome[1]{}

\DeclareMathOperator{\Log}{Log}

\begin{document}

\author{Bruno Angl\`es}

\address{
Normandie Université, 
Université de Caen Normandie - CNRS, 
Laboratoire de Mathématiques Nicolas Oresme (LMNO), UMR 6139,  
14000 Caen, France.
}
\email{bruno.angles@unicaen.fr}
\title{On the formal  Carlitz module}

\date{\today}


\setcounter{tocdepth}{1}

\parskip 2pt

\begin{abstract}
Let $L/\mathbb F_q(\theta)$ be a finite extension, and let $P$ be a monic irreducible polynomial in $\mathbb F_q[\theta].$ We show that the value at $1$ of the Carlitz-Goss $P$-adic zeta-function attached to $L/K$   is intimately  connected to the module of Taelman's units  (for the Carlitz module)  associated to finite extension  $L.$\end{abstract}



\maketitle

\tableofcontents
\section{Introduction}
Let $p$ be a prime number and let $F$ be a totally real number field. Let $\zeta_{p,F}$ be the $p$-adic Dedekind zeta function associated to $F.$ In 1988 P. Colmez proved the following remarkable $p$-adic analytic class number formula (\cite{COL}) :
$$ \lim_{s\rightarrow 1}(s-1) \zeta_{p,F}(s)= h_F \prod_{\mathfrak P\mid p }(1-\frac{1}{\mid \frac{O_F}{\mathfrak P}\mid})\,  \frac{2^{[F:\mathbb Q]-1}R_{p,F} }{ \sqrt{D_F}},$$
 where $h_F$ is the ideal class number of the ring of integers $O_F$ of $F,$ $\mathfrak P$ runs through the maximal ideals of $O_F$ above $p,$ $R_{p,F}$ is the $p$-adic regulator of $F,$ $D_F$  is the discriminant of $F.$ Our aim in these notes   is to show how a slight variation of Taelman's class formula (\cite{TAE2})  for the Carlitz module  enable us to prove an analogous formula for "real" extensions of $K:=\mathbb F_q(\theta).$\par
 
Let $A= \mathbb F_q[\theta], K=\mathbb F_q(\theta), K_\infty=\mathbb F_q((\frac{1}{\theta})).$ Let $L/ K$ be a finite extension, let $O_L$ be the integral closure of $A$ in $L,$ and let $L_\infty=L\otimes_K K_\infty.$ Lenny Taelman observed (see \cite{TAE1})  that the Carlitz exponential ($\exp_C$)  induces an exact sequence of $A$-modules :
$$0\rightarrow \frac{L_\infty}{\mathcal U(O_L)} \rightarrow \frac{L_\infty}{O_L}\rightarrow H(O_L)\rightarrow 0,$$
where $\mathcal U(O_L)=\{ x\in L_\infty, \exp_C(x)\in O_L\}$  is a  $A$-lattice in $L_\infty$ (i. e.  this $A$-module is free of rank $[L:K]$ and contains a $K_\infty$-basis of  $L_\infty$),  and $H(O_L)$ is a finite $A$-module, called Taelman's class module, and  this latter module can be viewed as an analogue of the ideal class group  of a number field. D. Goss proved  that the following infinite sum converges in $K_\infty $ (see paragraph \ref{section1}) :
$$\zeta_{O_L}(1):=\sum_{m\geq 0}\, \, \sum_{I,\, \dim_{\mathbb F_q}\frac{O_L}{I}=m}\, \, \frac{1}{[\frac{O_L}{I}]_A},$$ 
 where $I$ runs through the non-zero ideals of $O_L,$ and for a finite $A$-module $M,$ we denote by $[M]_A$ the monic generator of the Fitting ideal of the $A$-module $M.$ 
Taelman proved  the following  remarkable $\infty$-adic class formula  (\cite{TAE2}):
 $$\zeta_{O_L}(1)= [H(O_L)]_A\,  R_{\infty,L} ,$$
 where $R_{\infty, L}$ is a "regulator" associated to the $A$-lattice $\mathcal U(O_L).$  With the help of the notion of "$z$-deformation" of Anderson mlodules, Taelman's class formula has been generalized  to the case of "general" $A$ and for certain Anderson modules (\cite{ANDT2}).\par

Let $P$ be a monic irreducible polynomial in $A.$  We suppose that $L/ K$  is "real" (i.e. for each place $v$ of $L$ above $\infty,$ $\widetilde{\pi}$ is not contained in the $v$-adic completion of $L,$ where $\widetilde{\pi}$ is a period of the Carlitz module).  Let's consider the $A$-module of Taelman's units  :
$$U(O_L)=\exp_C(\mathcal U(O_L))= \exp_C(L_\infty)\cap O_L.$$
Since $L/K$ is real, $U(O_L)$ is a free $A$-module (via the Carlitz module) of rank $[L:K].$  Therefore, we can defined a $P$-adic regulator associated to $U(O_L)$ (see paragraph \ref{P_adic class}) :
$$R_{P,L}\in K_P,$$
where $K_P$ is the $P$-adic completion of $K.$ D. Goss  also proved  that the following infinite sum converges in $K_P $ (see paragraph \ref{section2}) :
$$\zeta_{P,O_L}(1):=\sum_{m\geq 0}\, \, \sum_{I,\, \dim_{\mathbb F_q}\frac{O_L}{I}=m}\, \, \frac{1}{[\frac{O_L}{I}]_A},$$ 
 where $I$ runs through the ideals of $O_L$ prime to $PO_L.$  Then, with the help of the theory of the $z$-deformation of the Carlitz module (see paragraph \ref{section3}),  we prove  the following $P$-adic class formula  (Theorem \ref{Theorem4}):
 $$\zeta_{P,O_L}(1)= [H(O_L)]_A\,  \prod_{\frak P\mid P} (1-\frac{1}{[\frac{O_L}{\frak P}]})\, R_{P,L} .$$
 The   Caen mathematician A. Lucas recently informed the author  that he is able to  generalize the above result  to the case of $\mathbb F_q[\theta]$-Anderson-modules. We refer the interested reader to the forthcoming work of Lucas.\par
 ${}$\par
The author thanks the A.N.R. project  PadLeFAn for its financial support; the author also thanks the referee for fruitful comments and suggestions.\par

\section{Notation}\label{section0}
 Let $\mathbb F_q$ be a finite field having $q$ elements and let $\theta$ be an indeterminate over $\mathbb F_q.$ Let $A=\mathbb F_q[\theta],$ $A_+=\{ a\in A, \, a {\rm \, monic}\, \}, $and for $n\in \mathbb N, $ $A_{+,n}=\{ a\in A_+,  \deg_\theta a=n\} .$\par
  Let $K=\mathbb F_q(\theta),$ $K_\infty=\mathbb F_q((\frac{1}{\theta})),$ and let  $\bar{K}_\infty$ be an algebraic closure of $K_\infty.$  Let $\bar K$ be the  algebraic closure of $K$ in $\bar{K}_\infty.$ We fix $^{q-1}\sqrt{-\theta}$ a $(q-1)$th root of $-\theta$ in $\bar{K}_\infty,$ and we set :
  $$\widetilde{\pi} = \theta \, ^{q-1}\sqrt{-\theta}\prod_{i\geq 1} (1-\theta^{1-q^i})^{-1}\in \bar{K}_\infty.$$
  The element $\widetilde{\pi}$ is called the Carlitz period  since it generates the $A$-module of the kernel of the Carlitz exponential (see \cite{GOS}, chapter 3).
 Finally let $z$ be an indeterminate over $K_\infty.$ \par
Let $L/K$ be a finite extension, $L\subset \bar{K}.$ We denote by $S_\infty(L)$ the set of places of $L$ above $\infty$ ($\infty$ is the unique place of $K$ which is a pole of $\theta$). Let $v\in S_\infty(L),$ then there exists a unique  $K$-embedding $\iota_v: L\rightarrow { \bar K}_\infty $ corresponding to the place $v,$ and we denote by $L_v$ the compositum of $K_\infty $ and $\iota_v(L)$ in ${\bar K}_\infty.$  We have an isomorphism of $K_\infty$-algebras :
$$L_\infty :=L\otimes_KK_\infty \simeq \prod_{v\in S_\infty(L)}L_v.$$
We will identify these two $K_\infty$-algebras and we view $L$ as a sub-$K$-algebra of $L_\infty$ via the diagonal embedding.
The extension $L/K$ is called  "real" if $\widetilde {\pi} \not \in L_v, \forall v\in S_\infty(L).$ Let $O_L$ be the integral closure of $A$ in $L.$ If $I$ is a non-zero ideal of $O_L,$ we denote by $n_{L/K}(I) \in A_+$ the monic polynomial which generates  the Fitting ideal of the finite $A$-module $\frac{O_L}{I}.$ 

\section{Carlitz-Goss $\infty$-adic Zeta values}\label{section1}

Let $L/K$ be a finite extension. For $m\in \mathbb N,$ we set :
$$J_m(O_L)=\{I\, {\rm  ideal\, of \,} O_L,I\not =\{0\},  \deg_\theta n_{L/K}(I)=m\}.$$
For example, if $L=K,$ then $O_L=A,$ and $J_m(A)= \{ aA, \, a\in A_{+,m}\}.$\par
For all $n\in \mathbb Z,$ we set :
$$Z_{O_L}(n;z)= \sum_{m\geq 0}(\sum_{I\in J_m(O_L)}n_{L/K}(I)^{-n})z^m \in K[[z]].$$
For example, if $L=K,$ then $O_L=A,$ and we have :
$$Z_A(n;z)=\sum_{m\geq 0}\sum_{a\in A_{+,m}}\frac{1}{a^n} z^m.$$
These power series were introduced by David Goss (see for example \cite{GOS}, chapter 8, paragraph 8.5). By \cite{GOS}, Theorem 8.9.2, we have that : $\forall \alpha \in \bar{K}_\infty, Z_{O_L}(n;z)\mid_{z=\alpha}$ converges in $\bar{K}_\infty.$ Furthermore :
\begin{lemma}\label{LemmaPol} Let  $n\in \mathbb Z, n\leq 0.$ We denote by $\ell_q(-n)$ the sum of the digits of $-n$ in base $q.$  Then $ Z_{O_L}(n;z)\in A[z], $ and  :
$$\deg_zZ_{O_L}(n;z) \leq [L:K] (2g_L+1+\ell_q(-n))-1,$$
where $g_L$ is the genus of $L.$
\end{lemma}
\begin{proof}
We give a proof of  this  crucial Lemma. We set :
$$N=[L:K](2g_L+1+\ell_q(-n)),$$ 
where $g_L$ is the genus of $L.$
Let $h$ be the number of ideal classes of $O_L.$  Let $I_1, \ldots, I_h \subset O_L$ be a system of representatives of ideals classes of $O_L.$ For $d\geq 0,$ lets' set :
$$U_d= \sum_{I\in J_d(O_L)}n_{L/K}(I)^{-n}\in A.$$
Let's set :
$$\forall i= 1, \ldots, h, S_{d,i}=\{ aO_L, a\in L^\times, aI_i\subset O_L, \deg_\theta(n_{L/K}(aI_i))=d\}.$$
 Let  $t\in \mathbb N\setminus\{0\},$ let $X_1, \ldots, X_t$ be $t$ indeterminates over $\mathbb F_q.$   Let $\rho_i : A\rightarrow \mathbb F_q[X_1, \ldots, X_t]$ be morphism of $\mathbb F_q$-algebras  such that $\rho_i(\theta) =X_i,$ for  $i=1, \ldots, t.$ By \cite{ANDT}, Lemma 3.2,  for $d\geq  [L:K] (2g_L+1+t),$  we have :
 $$\forall i=1, \ldots, h, \sum_{aO_L\in S_{d,i}}\rho_1(n_{L/K}(aO_L))\cdots \rho_t(\rho_1(n_{L/K}(aO_L))=0.$$
Now let's write $-n$ in base $q$ : $-n= q^{n_1}+\cdots +q^{n_{\ell_q(-n)}}, n_j \in \mathbb N.$ If we set $t=\ell_q(-n),$ and specialize $X_i$ in $\theta^{q^{n_i}},$ we get :
$$\forall d\geq N, \forall i=1, \ldots, h, \sum_{aO_L\in S_{d,i}}n_{L/K}(aO_L)^{-n}=0.$$
Let's observe that :
$$U_d= \sum_{i=1}^h n_{L/K}(I_i)^{-n}(\sum_{aO_L\in S_{d,i}}n_{L/K}(aO_L)^{-n}).$$
Thus for $d\geq N,$ $U_d=0.$
 We can therefore write :
 $$ Z_{O_L}(n;z)=\sum_{d\geq 0} U_dz^d \in A[z].$$
 \end{proof}
\begin{definition}\label{infinyzeta} The Carlitz-Goss $\infty-$adic zeta values are defined as follows  :
$$\forall n\in \mathbb Z, \zeta_{O_L}(n)=  Z_{O_L}(n;z)\mid_{z=1} \in K_\infty.$$
\end{definition}
\begin{remark}\label{remarkinfiny} Let's observe that : $$\forall n\geq 1, \zeta_{O_L}(n)\in 1+\frac{1}{\theta} \mathbb F_q[[\frac{1}{\theta}]],$$
$$\forall n\leq 0, \zeta_{O_L}(n) \in A.$$
Furthermore, by \cite{THA}, Theorem 5, we have :
$$\forall n\leq -1, n\equiv 0\pmod{q-1}, \zeta_{O_L}(n)=0.$$ 
\end{remark}
\section{Carlitz-Goss $P$-adic Zeta values}\label{section2}
Let $P$ be a prime of $A$ (i. e. $P$ is a monic irreducible polynomial in $A$), and let $K_P$ be the $P$-adic completion of $K.$ We fix $\bar{K}_P$ an algebraic closure of $K_P.$ Let $A_P$ be the valuation ring of $K_P.$ If $\mathbb F_P$ denotes the residue field of $A_P$ then $A_P=\mathbb F_P[[P]].$ \par
Let  $L/K$ be a finite extension. 
 We set :
 $$\forall n\in \mathbb Z, Z_{P,O_L}(n;z)= \sum_{d\in \mathbb N}U_{P,d}(n)z^d ,$$
 where : 
$$\forall d\in \mathbb N, \forall n\in \mathbb Z, U_{P,d}(n)= \sum_{I,\, \deg_\theta n_{L/K}(I)=d}n_{L/K}(I)^{-n}\in K,$$
where $I$ runs through the ideals of $O_L$ prime to $PO_L$ and such that $d= \dim_{\mathbb F_q}\frac{O_L}{I}= \deg_\theta n_{L/K}(I).$
Let's observe that :
\begin{lemma}\label{Lemma1} Let $n\leq 0$ be an integer. Then  $Z_{P,O_L}(n;z)\in A[z],$ and :
$$\deg_zZ_{P,O_L}(n; z) \leq  [L:K](2g_L+1+\ell_q(-n))+\sum_{\mathfrak P\mid PO_L}\deg_\theta n_{L/K}(\mathfrak P)-1,$$ 
where $\ell_q(-n)$ is the sum of digits of $n$ in base $q.$
\end{lemma}
\begin{proof} By Lemma \ref{LemmaPol}, we have : $ \deg_z Z_{O_L}(n;z)\leq [L:K](2g_L+1+\ell_q(-n))-1.$ We have  :
$$ Z_{P,O_L}(n;z)= \prod_{\mathfrak P \mid PO_L}(1-n_{L/K}(\mathfrak P)^{-n}z^{\deg_\theta (n_{L/K}({\mathfrak P})})Z_{O_L}(n;z) .$$
 The Lemma follows.
\end{proof}

\begin{lemma}\label{Lemma2} Let $n\geq 1$ be an integer, and let $s\in \mathbb N$ such that  $n\leq q^s-1.$   Let $d\geq [L:K](2g_L+1+(s+1+\deg_\theta P)(q-1))+\sum_{\mathfrak P\mid PO_L}\deg_\theta n_{L/K}(\mathfrak P).$ Then:
$$v_P(U_{P,d}(n))\geq q^{s+1}.$$
In particular, if $d\geq [L:K](2g_L+1+(1+\deg_\theta P)(q-1))+\sum_{\mathfrak P\mid PO_L}\deg_\theta n_{L/K}(\mathfrak P),$ we have   :
$$v_P(U_{P,d}(n))\geq q^{\frac{d-\sum_{\mathfrak P\mid PO_L\deg_\theta n_{L/K}(\mathfrak P)}}{[L:K](q-1)}-\frac{2g_L+1}{q-1}-\deg_\theta P},$$
wher $v_P$ is the usual $P$-adic valuation on $K_P.$
\end{lemma}
\begin{proof}  Write :
$$n=\sum_{i=0}^s n_i q^i, n_i\in \{ 0, \ldots, q-1\}.$$
We set $m= n+q^{1+s}( q^{\deg_{\theta} P}-1),$ we have :
$$\ell_q(m)= \sum_{i=0}^s n_i + \deg_\theta P (q-1) \leq(s+1+\deg_\theta P) (q-1).$$
Thus, by Lemma \ref{Lemma1}, if $d\geq  [L:K](2g_L+1+(s+1+\deg_\theta P)(q-1))+\sum_{\mathfrak P\mid PO_L}\deg_\theta n_{L/K}(\mathfrak P),$ we have :
$$U_{P,d}(-m)=0.$$
Furthermore, observe that :
$$\forall k\geq 0, v_P(U_{P,k}(n)-U_{P,k}(-m))\geq q^{s+1}.$$
Thus :
$$v_P(U_{P,d}(n))\geq q^{s+1}.$$
Now, let  $d\in \mathbb N$ such that $d\geq [L:K](2g_L+1+(1+\deg_\theta P)(q-1))+\sum_{\mathfrak P\mid PO_L}\deg_\theta n_{L/K}(\mathfrak P)$ and let $s\in \mathbb N$ be an integer such that :
$$\frac{1}{q-1}(\frac{d-\sum_{\mathfrak P\mid PO_L}\deg_\theta n_{L/K}(\mathfrak P)}{[L:K]}-2g_L-1)-1-\deg_\theta P\leq s.$$
Then $$v_P(U_{P,d}(n))\geq q^{s+1}\geq q^{\frac{d-\sum_{\mathfrak P\mid PO_L\deg_\theta n_{L/K}(\mathfrak P)}}{[L:K](q-1)}-\frac{2g_L+1}{q-1}-\deg_\theta P}.$$ 
The Lemma follows.
\end{proof}
The above Lemma implies that: $\forall n\in \mathbb Z, $ $\forall \alpha \in \bar{K}_P, Z_{P,O_L}(n;z)\mid_{z=\alpha}$ converges in $\bar{K}_P.$ 
\begin{definition} \label{P-adiczeta} The Carlitz-Goss $P-$adic zeta values are defined as follows  :
$$\forall n\in \mathbb Z, \zeta_{P, O_L}(n)=  Z_{P,O_L}(n;z)\mid_{z=1} \in A_P.$$
\end{definition}
Let's observe that :
$$\forall n\leq 0, \zeta_{P, O_L}(n)= \prod_{\mathfrak P \mid PO_L}(1-n_{L/K}(\mathfrak P)^{-n}) \, \zeta_{O_L}(n) \in A.$$
Furthermore :
\begin{lemma}
$$\forall n\geq 0, n\equiv 0\pmod{q-1}, \zeta_{P, O_L}(n)=0.$$
\end{lemma}
\begin{proof} Let $k\in \mathbb N$ and set  $m_k:= n-q^{k+n} (q^{\deg_\theta P}-1)\leq -1.$ We have : 
 $$\forall d\in \mathbb N, v_P(U_{P,d}(n)-U_{P,d}(m_k))\geq q^{n+k}.$$
Thus :
$$v_P(\zeta_{P, O_L}(n)-\zeta_{P, O_L}(m_k))\geq q^{k+n}.$$
But, since $m_k\leq 0,$ we have  :
$$\zeta_{P, O_L}(m_k)= \prod_{\mathfrak P \mid PO_L}(1-n_{L/K}(\mathfrak P)^{-m_k})\zeta_{O_L}(m_k).$$
Therefore, by by \cite{THA}, Theorem 5, we get :
$$\zeta_{P, O_L}(m_k)= 0.$$
Thus:
$$\forall k\in \mathbb N, v_P(\zeta_{P, O_L}(n))\geq q^{k+n}.$$
The Lemma follows.
\end{proof}
\section{Taelman's class formula}\label{section3}

\subsection{The Carlitz module}\label{Carlitz} We recall some basic facts on the Carlitz module (see for example \cite{GOS}, chapter 3). \par
Let $K_\infty\{\{\tau\}\}$ be the non-commutative ring of power series in the variable $\tau$ with coefficients in $K_\infty,$  where the addition is the usual one and the multiplication rule is given by :
$$\forall x\in K_\infty, \tau x= x^q \tau.$$
We denote by $K_\infty\{\tau\}$ the sub-ring of $K_\infty\{\{\tau\}\}$ whose elements are the polynomials in $\tau$ with coefficients in $K_\infty.$ \par
 The Carlitz module is the morphism of $\mathbb F_q$-algebras $C:A\rightarrow K_\infty\{\tau\}$ given by :
 $$C_\theta =\theta +\tau.$$ Let $a\in A\setminus\{0\},$ we write :
 $$C_a=\sum_{i=0}^{\deg_\theta a} [a,i] \tau ^i, \, [a,i] \in A.$$
One can prove  that there exists a unique power series $\exp_C \in 1+K_\infty\{\{\tau\}\}\tau$ such that :
$$\exp_C\theta =C_\theta \exp_C.$$
 In fact :
 $$\exp_C=\sum_{i\geq 0} \frac{1}{D_i} \tau^i,$$
 where $D_0=1$ and for $i\geq 1,$ $D_i= \prod_{k=0}^{i-1} (\theta^{q^i}-\theta^{q^k}).$ Furthermore there exists a unique power series $\Log_C \in 1+K_\infty\{\{\tau\}\}\tau$ such that :
 $$\Log_C C_\theta= \theta \Log_C.$$
  We have :
  $$\Log_C= \sum_{i\geq 0} \frac{1}{ L_i} \tau ^i,$$
  where $L_0=1, $ and for $i\geq 1, $  $ L_i =(\theta-\theta^q)\cdots (\theta -\theta^{q^i}).$  Let's observe that :
  $$\exp_C \Log_C =\Log_C \exp_C.$$
  ${}$\par
  Let $M$ be a $A$-module equipped with a $\mathbb F_q$-linear map $\tau : M\rightarrow M$ such that $\forall a \in A, \forall m\in M, \tau (am)=a^q\tau (m).$ Then the Carlitz module induces a new structure of $A$-module on $M:$
  $$\forall a\in A, \forall m\in M, a.m =C_a(m).$$
  We denote by $C(M)$ the $A$-module $M$ equipped with the structure of $A$-module induced by $C.$\par
  
  Let $v_\infty : \bar{K}_\infty \rightarrow \mathbb Q\cup \{+\infty\}$ be the usual $\infty$-adic valuation on $\bar{K}_\infty$ normalized such that $v_\infty (\theta)=-1.$ Let $D_\infty =\{ x\in \bar{K}_\infty, v_\infty(x)>\frac{-q}{q-1}\}\subset \bar{K}_\infty.$    
  Let $\tau :\bar{K}_\infty\rightarrow \bar K_\infty, x\mapsto x^q.$ 
   Let $i\geq 0$ and $x\in \bar{K}_\infty,$ then :
  $$v_\infty (\frac{\tau^i(x)}{D_i})=v_\infty(\frac{x^{q^i}}{D_i})= q^i (v_\infty(x) +i).$$
  This implies that $\exp_C$ converges on $\bar{K}_\infty$ and :
  $$\forall x\in D_\infty, v_\infty(\exp_C(x))=v_\infty(x).$$
   Furthermore, $\exp_C$ induces a short exact sequence of $A$-modules :
   $$0\rightarrow \widetilde{\pi}A\rightarrow \bar{K}_\infty\rightarrow C(\bar{K}_\infty)\rightarrow 0.$$
   Let's observe that :
  $$D_\infty \cap A= \mathbb F_q\oplus \mathbb F_qC_\theta (1).$$
  Let $i\geq 0$ and $x\in \bar{K}_\infty,$ then :
  $$v_\infty(\frac{x^{q^i}}{L_i})= q^i (v_\infty(x) +\frac{q}{q-1})-\frac{q}{q-1}.$$
This implies that $\Log_C$ converges on $D_\infty$ and :
  $$\forall x\in D_\infty, v_\infty(\Log_C(x))=v_\infty(x).$$
Furthermore :
$$\forall x\in D_\infty, \exp_C(\Log_C(x))=\Log_C(\exp_C(x))=x.$$
In particular $\exp_C, \Log_C  : D_\infty \rightarrow D_\infty$ are isometric  $\mathbb F_q$-linear maps. Finally, let's observe that :
\begin{lemma}\label{LemmaTors}
$$C(\bar{K}_\infty) _{\rm Tors} \subset D_\infty.$$
\end{lemma}
\begin{proof} We have :
$$C(\bar{K}_\infty) _{\rm Tors} = \exp_C (\widetilde{\pi}K).$$
Let $\lambda \in C(\bar{K}_\infty) _{\rm Tors}, $ since $\exp_C(\widetilde{\pi}A)=\{0\},$ there exist $a,b\in a,$ $b\not = 0,$ $\deg_\theta a< \deg_\theta b,$ such that :
$$\lambda =\exp_C(\widetilde{\pi} \frac{a}{b}).$$
But $v_\infty (\widetilde{\pi} \frac{a}{b})= v_\infty (\widetilde{\pi})+\deg_\theta b-\deg-\theta a =-\frac{q}{q-1}+\deg_\theta b-\deg-\theta a>-\frac{q}{q-1}.$ Thus $\widetilde{\pi} \frac{a}{b}\in D_\infty,$ and therefore $\lambda \in \exp_C(D_\infty)=D_\infty.$
\end{proof}

\subsection{Units and class modules}\label{Taelman} We briefly describe the main results obtained by L. Taelman in the case of the Carlitz module (see \cite{TAE1}, \cite{TAE2}).\par
Let $M$ be a finite $A$-module, we denote by $[M]_A\in A_+$ the monic generator of the Fitting ideal of the $A$-module $M.$\par
Let $L/K$ be a finite extension.
Let us begin by the following observation (\cite{TAE2}, Proposition 1 and Example 1)  which is a key ingredient for the proof of Taelman's class formula :
\begin{lemma}\label{Lemma3} We have :
$$\zeta_{O_L}(1)=\prod_{\mathfrak P} \frac{[ \frac{O_L}{\mathfrak P}]_A}{[ C(\frac{O_L}{\mathfrak P})]_A},$$
where $\mathfrak P$ runs through the maximal ideals of $O_L.$ More precisely,  let $\mathfrak P$ be a maximal ideal of $O_L,$ then :
$$[C(\frac{O_L}{\mathfrak P})]_A=n_{L/K}(\mathfrak P)-1.$$
\end{lemma}
\begin{proof} We give a proof for the convenience of the reader. \par
\noindent Let  $\mathfrak P$ be a maximal ideal of $O_L,$ and $P$ be the prime of $A$ such that $\mathfrak P \cap A=PA.$ Recall that  $C_P=\sum_{i=0}^d [P,i]\tau ^i,  [P,i]\in A,$ where $d=\deg_\theta P.$ Then $[P,0]=P$ and $[P,d]=1.$
From the equation :
$$C_P C_\theta =C_\theta C_P,$$ 
we get (see \cite{GOS}, Proposition 3.3.10) :
$$\forall  i=1, \ldots, d,\,  [P,i]= \frac{[P,i-1]^q-[P,i-1]}{\theta^{q^i}-\theta}.$$
This implies that :
$$\forall  i=0, \ldots, d-1, \, [P,i]\equiv 0\pmod{P}.$$ 
 Write $n_{L/K} (\mathfrak P) =P^f, $ where $f=\dim_{\frac{A}{PA} } \frac{O_L}{\mathfrak P}.$ Then :
 $$\forall x\in O_L, C_{P^f}(x) \equiv x\pmod{\mathfrak P}.$$
  Let $a$ be the monic generator of the annihilator of the $A$-module $C(\frac{O_L}{\mathfrak P}),$ then $a$ divides $ n_{L/K}(\mathfrak P)-1.$ We observe that the number of elements  $\alpha$ in an algebraic closure of $\frac{A}{PA}$ such that $C_a(\alpha)=0$ is less than $q^{\deg_\tau C_a}.$ Since $\frac{O_L}{\mathfrak P}$ is a finite extension of $\frac{A}{PA},$ we get :
  $$\dim_{\mathbb F_q} \frac{O_L}{\mathfrak P}\leq \deg_\tau C_ a= \deg_\theta a.$$
  Note that :
  $$\dim_{\mathbb F_q} \frac{O_L}{\mathfrak P}= \deg_\theta n_{L/K}(\mathfrak P)=\deg_\theta ( n_{L/K}(\mathfrak P)-1).$$
  We deduce that $a= n_{L/K}(\mathfrak P)-1$ and that we have an isomorphism of $A$-modules:
  $$C( \frac{O_L}{\mathfrak P})\simeq \frac{A}{ (n_{L/K}(\mathfrak P)-1)A}.$$ 
\end{proof}
Recall that  $L_\infty=L\otimes_KK_\infty$ (see section \ref{section0}). Let's set :
$$\Lambda_L =\oplus_{v\in S_\infty(L)} \widetilde{\pi}_v A,$$
 where $\widetilde{\pi}_v =\widetilde \pi$ if $\widetilde \pi \in L_v,$ and $\widetilde{\pi}_v =0$ otherwise.
Then, by paragraph \ref{Carlitz},   $\exp_C$ converges on $L_\infty$ and gives rise to a short exact sequence of $A$-modules :
$$0\rightarrow \Lambda_L \rightarrow L_\infty \rightarrow C(L_\infty).$$
Recall that the extension $L/K$ is called "real" if $\forall v\in S_\infty(L),$ $\widetilde{\pi} \not \in L_v.$ Therefore  $L/K$ is real  if and only if $\exp_C$ is injective on $L_\infty.$ Let's set :
$$U(O_L)=C(O_L)\cap \exp_C(L_\infty),$$
$$\mathcal U(O_L)=\{ x\in L_\infty, \exp_C(x)\in O_L\}.$$
Then $\exp_C$ induces a short exact sequence of $A$-modules :
$$0\rightarrow \Lambda_L \rightarrow \mathcal U(O_L)\rightarrow U(O_L)\rightarrow 0.$$
L. Taelman proved that $\mathcal U(O_L)$ is $A$-lattice in $L_\infty$ (\cite {TAE2}, Proposition 3)   : $\mathcal U(O_L)$ is a free $A$-module of rank $[L:K]$ and $L_\infty$ coincides with the $K_\infty$-space generated by $\mathcal U(O_L).$ In particular $U(O_L)$ is a finitely generated sub-$A$-module of $C(O_L).$ We will call $U(O_L)$ the module of "units" of $C/O_L.$\par
We have :
\begin{lemma}\label{Lemma4} ${}$\par
\noindent 1) $C(O_L)_{\rm Tors} \subset U(O_L).$\par
\noindent 2) $C(O_L)_{\rm Tors}\not =\{0\}$ if and only if ${\rm rank}_A \Lambda_L=\mid S_\infty(L)\mid.$\par
\noindent 3) $U(O_L)=C(O_L)_{\rm Tors}$ if and only if $q=2$ and $\infty$ splits totally in $L/K.$\end{lemma}
\begin{proof}  Let $\lambda \in C(O_L)_{\rm Tors} \setminus \{0\}.$ Let $v\in S_\infty(L),$ let $\iota_v: L\rightarrow \bar{K}_\infty$ be the $K$-embedding associated to $v$ (see section \ref{section0}). Then :
$$K_\infty(\iota_v(\lambda))\subset L_v.$$
But $\iota_v(\lambda)\in D_\infty$ (Lemma \ref{LemmaTors}),  thus :
$$\Log_C(\iota_v(\lambda))\in L_v,$$ and we obtain :
$$\widetilde{\pi} \in L_v,$$
$$\iota_v(\lambda)\in \exp_C(\widetilde{\pi}K)\subset \exp_C(L_v).$$ 
We get 1) and 2) (see also \cite{TAE1}, Proposition 2 for a proof of the first assertion). For the last assertion. We have $U(O_L)=C(O_L)_{\rm Tors}$  if and only if ${\rm rank}_A \Lambda_L ={\rm rank}_A\mathcal U(O_L) = [L:K].$ Thus if $U(O_L)=C(O_L)_{\rm Tors}$ we have :
$$[L:K]= {\rm rank}_A \Lambda_L \leq \mid S_\infty(L)\mid .$$
Therefore  $\mid S_\infty(L)\mid= [L:K]$ and thus $\infty$ splits totally in $L/K.$ Furthermore we must have $\widetilde{\pi}\in K_\infty$ which implies $q=2.$ The Lemma follows.
\end{proof} 
Let's set :
$$H(O_L)=\frac{C(L_\infty)}{C(O_L)+\exp_C(L_\infty)}.$$
Then $H(O_L)$ is a  finite $A$-module  (\cite{TAE2}, Proposition 5). We call $H(O_L)$ the " class module " of $C/O_L.$
We have a short exact sequence of $A$-modules :
$$0\rightarrow  \frac{C(O_L)}{U(O_L)} \rightarrow \frac{C(L_\infty)}{\exp_C(L_\infty)}\rightarrow H(O_L)\rightarrow 0.$$

Let $sgn : K_\infty^\times \rightarrow \mathbb F_q^\times$ be the usual sign function, i.e. :
$$sgn(\sum_{i\geq i_0} \alpha_i ( \frac{1}{\theta})^i)= \alpha_{i_0}, \, i_0 \in \mathbb Z, \alpha_i \in \mathbb F_q, \alpha_{i_0}\not =0.$$
Then $A_+=\{ a\in A, sgn(a)=1\}.$ Let $\sigma : L_\infty\rightarrow L_\infty$ the $K_\infty$-linear map  which sends an $A$-basis of  $O_L$ to an $A$-basis of $\mathcal U(O_L).$ We define the "regulator" of $C/O_L$ to be :
$$R_L = \frac{\det (\sigma )}{sgn(\det (\sigma))}\in K_\infty ^{\times}.$$
Let's observe that $R_L$ is independent of the choices of an $A$-basis of $O_L$ and an $A$-basis of $\mathcal U(O_L).$
Then L. Taelman proved the following remarkable equality  (\cite{TAE2}, Theorem 1) :
$$\zeta_{O_L}(1)= R_L\,  [H(O_L)]_A.$$
\begin{example}Let's finish this paragraph by a basic example : $L=K$  (\cite{TAE2}, Example 3). We have :
$$K_\infty =A\oplus \frac{1}{\theta} \mathbb F_q[[\frac{1}{\theta}]].$$
Since $1\in D_\infty,$ we get :
$$K_\infty= \Log_C(1)A\oplus  \frac{1}{\theta} \mathbb F_q[[\frac{1}{\theta}]].$$
We obtain :
$$\exp_C(K_\infty) =U(A)\oplus   \frac{1}{\theta} \mathbb F_q[[\frac{1}{\theta}]],$$
and $U(A)$ is the sub-$A$-module of $C(A)$ generated by $1.$ Note that $U(A)$ is a free $A$-module of rank one if $q\geq 3,$ and if $q=2,$ $U(A)$ is exactly the set of $\theta^2+\theta$-torsion points of $C$ (in particular, as an $A$-module,  $U(A)$ is isomorphic to $\frac{A}{(\theta^2+\theta)A}$). We also obtain :
$$H(A)=\{0\}, \mathcal U(A)=\Log_C(1) A.$$
Since $sgn(\Log_C(1))=1,$ we get :
$$\zeta_A(1)=R_A=\Log_C(1).$$
Now, for $q=2,$ we also get :
$$\widetilde{\pi} = (\theta^2+\theta) \Log_C(1).$$
\end{example}

\subsection{$\infty$-adic $z$-deformation}\label{z-deformation}
We present some basic constructions attached to the Carlitz module. These constructions were introduced in \cite{ATR}  in the more general context of Drinfeld modules (see also \cite{DEM}, \cite{ANDT2}, \cite{BEA}). This paragraph  has its origins in the works of G. Anderson and D. Thakur  on Log-algebraicity (\cite{ANDTHA}, \cite{AND}, see also \cite{PAP}).\par
We set  $\mathbb A=\mathbb F_q(z)[\theta],$ $\mathbb K=\mathbb F_q(z)(\theta),$ $\mathbb K_\infty =\mathbb F_q(z)((\frac{1}{\theta})).$ Let $sgn :\mathbb K_\infty^\times\rightarrow  \mathbb F_q(z)^\times$ be the usual sign function, i.e. :
$$sgn(\sum_{i\geq i_0} \alpha_i ( \frac{1}{\theta})^i)= \alpha_{i_0}, i_0\in \mathbb Z, \alpha_i\in \mathbb F_q(z), \alpha_{i_0}\not =0.$$
An element $a\in \mathbb A$ such that $sgn(a)=1$ is called monic. Let $\mathbb T (K_\infty) =\mathbb F_q[z]((\frac{1}{\theta}))$ which is a principal ideal domain.\par
Let $\tau :\mathbb K_\infty \rightarrow \mathbb K_\infty$ be the continuous morphism of $\mathbb F_q(z)$-algebras such that :
$$\forall x\in K_\infty, \tau(x)=x^q.$$
 Let $\mathbb K_\infty\{\{\tau\}\}$ be the non-commutative ring of power series in the variable $\tau$ with coefficients in $\mathbb K_\infty,$  where the addition is the usual one and the multiplication rule is given by :
$$\forall x\in \mathbb K_\infty, \tau x= \tau(x) \tau.$$
We denote by $\mathbb K_\infty\{\tau\}$ the sub-ring of $\mathbb K_\infty\{\{\tau\}\}$ whose elements are the polynomials in $\tau$ with coefficients in $\mathbb K_\infty.$ \par
 The $z$-deformation of the Carlitz module is the morphism of $\mathbb F_q(z)$-algebras $\widetilde{C}:\mathbb A\rightarrow \mathbb K_\infty\{\tau\}$ given by :
 $$\widetilde{C}_\theta =\theta +z\tau.$$
There exists a unique power series $\exp_{\widetilde{C}} \in 1+\mathbb K_\infty\{\{\tau\}\}\tau$ such that :
$$\exp_{\widetilde{C}}\theta =\widetilde{C}_\theta \exp_{\widetilde{C}}.$$
 In fact :
 $$\exp_{\widetilde{C}}=\sum_{i\geq 0} \frac{z^i}{D_i} \tau^i,$$
Furthermore there exists a unique power series $\Log_{\widetilde{C}} \in 1+\mathbb K_\infty\{\{\tau\}\}\tau$ such that :
 $$\Log_{\widetilde{C}} \widetilde{C}_\theta= \theta \Log_{\widetilde{C}}.$$
  We have :
  $$\Log_{\widetilde{C}}= \sum_{i\geq 0} \frac{z^i}{ L_i} \tau ^i.$$
  Note that :
  $$\exp_{\widetilde{C}} \Log_{\widetilde{C}} =\Log_{\widetilde {C}} \exp_{\widetilde{C}}.$$
  ${}$\par
  Let$M$ be a $\mathbb A$-module.\par
 If $M$ is a finite dimensional $\mathbb F_q(z)$-vector space, we denote by $[M]_{\mathbb A}$ the monic generator of the Fitting ideal of the $\mathbb A$-module $M.$\par
Let  $M$ be  equipped with a $\mathbb F_q(z)$-linear map $\tau : M\rightarrow M$ such that $\forall a \in A, \forall m\in M, \tau (am)= a^q \tau (m).$ Then the $z$-deformation of the Carlitz module induces a new structure of $\mathbb A$-module on $M:$
  $$\forall a\in \mathbb A, \forall m\in M, a.m =\widetilde{C}_a(m).$$
 We denote by $\widetilde{C}(M)$ the $\mathbb A$-module $M$ equipped with the structure of $A$-module induced by $\widetilde{C}.$\par
 ${}$\par
Let $L/K$ be a finite extension, we set :
$$\mathbb L_\infty =L\otimes_K \mathbb K_\infty,$$
$$\mathbb T(L_\infty)= L\otimes_K \mathbb T(K_\infty).$$
Let $\mathbb O_L$ be the sub-$\mathbb F_q(z)$-algebra of $\mathbb L_\infty$ generated by $O_L.$\par
Let $M$ be a sub-$\mathbb A$-module of $\mathbb L_\infty.$  We say that $M$ is an $\mathbb A$-lattice in $\mathbb L_\infty$ if $M$ is a free $\mathbb A$-module of rank $[L:K]$ and if $M$ generates $\mathbb L_\infty$ as a $\mathbb K_\infty$-vector space. \par
Let $M$ be an $\mathbb A$-lattice in $\mathbb L_\infty,$ and let $\sigma :\mathbb L_\infty\rightarrow \mathbb L_\infty$ be a $\mathbb K_\infty$-linear map that sends an $A$-basis of $O_L$ to an $\mathbb A$-basis of $M,$ the "regulator" of $M$ is defined by (see \cite {ATR}, paragraph 2.1, for more details)  :
$$ R_M=\frac{\det \sigma}{sgn(\det(\sigma))} \in \mathbb K_\infty ^\times.$$

\begin{lemma}\label{Lemma5} We have :
$$Z_{O_L}(1;z) =\prod_{\mathfrak P}\frac{[\frac{\mathbb O_L}{\mathfrak P\mathbb O_L}]_{\mathbb A}}{[\widetilde{C}(\frac{\mathbb O_L}{\mathfrak P\mathbb O_L})]_{\mathbb A}} \in \mathbb T(K_\infty)^\times.$$
More precisely, let $\mathfrak P$ be a maximal ideal of $O_L$ then, we have an isomorphism of $\mathbb A$-modules :
$$\widetilde{C}(\frac{\mathbb O_L}{\mathfrak P\mathbb O_L})\simeq \frac{\mathbb A}{(n_{L/K} (\mathfrak P)-z^{\deg_\theta (n_{L/K} (\mathfrak P))})\mathbb A}.$$
In particular :
$$[\widetilde{C}(\frac{\mathbb O_L}{\mathfrak P\mathbb O_L})]_{\mathbb A}= n_{L/K} (\mathfrak P)-z^{\deg_\theta (n_{L/K} (\mathfrak P))}.$$
\end{lemma}
\begin{proof} Let $P$ be the prime of $A$ under $\mathfrak P.$ By the proof of Lemma \ref{Lemma3}, we have :
$$\widetilde{C}_{n_{L/K} (\mathfrak P)}\equiv z^{\deg_\theta (n_{L/K} (\mathfrak P))}\tau^{\deg_\theta (n_{L/K} (\mathfrak P))}\pmod{PA[z]\{\tau\}}.$$
This implies :
$$\widetilde{C}_{n_{L/K} (\mathfrak P)-z^{\deg_\theta (n_{L/K} (\mathfrak P))}}(\frac{\mathbb O_L}{\mathfrak P\mathbb O_L})=\{0\}.$$
Now, using \cite{APTR}, Lemma 5.7, we conclude as in the proof of Lemma \ref{Lemma3}.

\end{proof}

\begin{lemma}\label{Lemma6} We have :
$${\rm Ker} (\exp_{\widetilde{C}} : \mathbb L_\infty \rightarrow \mathbb L_\infty )=\{0\},$$
$$\mathbb L_\infty =\mathbb O_L +\exp_{\widetilde{C}}(\mathbb L_\infty).$$

\end{lemma}
\begin{proof}  Let us observe that $\exp_{\widetilde{C}}: \mathbb T(L_\infty)\rightarrow \mathbb T(L_\infty)$ is an injective morphism of $\mathbb F_q[z]$-modules. Since the $\mathbb F_q(z)$-algebra generated by $\mathbb T(L_\infty)$ is dense in $\mathbb L_\infty,$  we deduce that $\exp_{\widetilde{C}}$ is injective on $\mathbb L_\infty.$
The second assertion is a consequence of \cite{ATR}, Proposition 2.
\end{proof}
Let's set :
$$\mathcal U(\mathbb O_L)=\{ x\in \mathbb L_\infty, \exp_{\widetilde{C}}(x) \in \mathbb O_L\}.$$
Then $\mathcal U(\mathbb O_L)$ is an $\mathbb A$-lattice in $\mathbb L_\infty$ (see for example, \cite{DEM}, Proposition 2.8). Furthermore, we have (see for example, \cite{DEM}, Theorem 2.9) :
$$Z_{O_L}(1;z) =R_{\mathcal U(\mathbb O_L)}.$$
Let's set :
$$\mathcal U(O_L[z])=\{ x\in \mathbb T(L_\infty), \exp_{\widetilde{C}}(x) \in  O_L[z]\}.$$
We mention that we have identified  $O_L[z]$ with $\mathbb F_q[z]\otimes _{\mathbb F_q} O_L,$ and  $\mathbb O_L$ with  $\mathbb F_q(z) \otimes_{\mathbb F_q} O_L.$
\begin{lemma}\label{Lemma7}  The set $\mathcal U(O_L[z])$ is  finitely generated $A[z]$-module, and it generates $\mathcal U(\mathbb O_L)$ as a $\mathbb F_q(z)$-vector space.
\end{lemma}
\begin{proof}  We have :
$$\mathbb O_L \cap \mathbb T(L_\infty) =O_L[z].$$
Thus :
$$\mathcal U(\mathbb O_L)\cap \mathbb T(L_\infty)\subset \mathcal U(O_L[z]).$$
By \cite{ATR}, Proposition 1, $\mathcal U(\mathbb O_L)\cap \mathbb T(L_\infty)$ is a finitely generated $A[z]$-module and generates $\mathcal U(\mathbb O_L)$ as a $\mathbb F_q(z)$-vector space. The Lemma follows.
\end{proof}
Let's set :
$$\mathcal U_{St}(O_L)=\mathcal U(O_L[z])\mid_{z=1}\subset \mathcal U(O_L)\subset  L_\infty.$$
Then, by \cite{ATR}, section 2.5, we have that $\frac{\mathcal U(O_L)}{\mathcal U_{St}(O_L)}$ is a finite $A$-module and :
$$[\frac{\mathcal U(O_L)}{\mathcal U_{St}(O_L)}]_A=[H(O_L)]_A.$$
We set :
$$U(O_L[z])=\exp_{\widetilde{C}}(\mathcal U(O_L[z]))\subset O_L[z].$$

\begin{lemma}\label{Lemma8} ${}$\par
\noindent 1)  $\mathcal U(O_L [z])\subset L[[z]],$\par
\noindent 2) $\Log_{\widetilde{C}}$  converges on $U(O_L[z])$ and induces an $\mathbb F_q[z]$-isomorphism  $\log_{\widetilde{C}}: U(O_L[z])\rightarrow \mathcal U(O_L[z]).$\par
\end{lemma}
\begin{proof}   Let's recall that $\mathbb T(L_\infty)$ is the subring of $L_\infty[[z]]$ consisting of power series $\sum_{i\geq 0} x_i z^i, x_i \in L_\infty,$ such that $\lim_{i\rightarrow +\infty} x_i =0.$ Let's observe that $\exp_{\widetilde{C}}: L_\infty [[z]] \rightarrow \widetilde{C}(L_\infty [[z]])$ is an isomorphism of $A[[z]]$-modules, and for all $F\in L_\infty [[z]]:$
$$F= \exp_{\widetilde{C}}(\Log_{\widetilde{C}}(F)) = \Log_{\widetilde{C}}(\exp_{\widetilde{C}}(F)).$$
We have :
$$\mathcal U(O_L [z])=\log_{\widetilde{C}}(U(O_L[z]))\subset L[[z]]\cap \mathbb T(L_\infty).$$
The Lemma follows.
\end{proof}

Let $n=[L:K],$ we have  a refined  version of the class formula  :
  
  \begin{theorem}\label{Theorem1} Let $a_1, \ldots, a_n \in U(O_L[z])
  \subset O_L[z], $ and let $ M \in M_n (K[[z]])\cap M_n(\mathbb K_\infty)\subset M_n(\mathbb T(K_\infty))$ be the  $n\times n$ matrix obtained by  expressing $\log_{\widetilde{C}}(a_i)$ in a fixed $A$-basis of $O_L,$ we have :
 $$\det (M) =\beta  Z_{O_L}(1;z),$$
 where  $ \beta\in  A[z].$ Furthermore, if $N=\sum_{i=1}^n \log_{\widetilde{C}}(a_i)\mid_{z=1}A$ is an $A$-lattice in $L_\infty,$ then $ \frac{\mathcal U_{St}(O_L)}{N} $ is a finite $A$-module and ${\beta (1) }{[\frac{\mathcal U_{St}(O_L)}{N}]_A}^{-1}\in \mathbb F_q^\times.$     \end{theorem}
 \begin{proof}
By Lemma \ref{Lemma7}, there exist $u_1, \ldots, u_n \in \mathcal U(O_L[z])$ such that  :
 $$\mathcal U(\mathbb O_L)= \oplus_{i=1}^nu_i \mathbb A.$$
 We set :
$$b_i =\exp_{\widetilde{C}}(u_i)\in O_L[z].$$
Then :
$$u_i=\Log_{\widetilde{C}}(b_i) \in L[[z]]\cap \mathbb T(L_\infty).$$
Thus  :
  $$ \oplus_{i=1}^{n}u_i K[[z]]\subset L[[z]].$$
Let $\sigma : \mathbb L_\infty \rightarrow \mathbb L_\infty $  the $\mathbb K_\infty$-linear map  that sends a fixed $A$-basis of $O_L$ to  $(u_1, \ldots , u_n).$   Let $M'$ be the matrix of $\sigma$  relative to the $A$-basis of $O_L.$ Then :
  $$M'\in M_n(K[[z]])\cap M_n(\mathbb K_\infty)\subset M_n (\mathbb T(K_\infty)) .$$
  By the class formula for $\widetilde{C}/\mathbb O_L,$ , we have :
   $$\det (M') =sgn(\det M') \, Z_{O_L}(1;z) ,$$
   where  $\alpha := sgn(\det(M'))\in \mathbb F_q[z] \setminus \{0\}.$ 
  
  Now let $v_1, \cdots , v_n \in \mathcal {U}(O_L[z])$ such that  $N= \sum_{i=1}^nv_i\mid_{z=1} A$ is an $A$-lattice in $L_\infty.$  
   Then $v_1, \ldots, v_n$ are $\mathbb A$-linearly independent.  Furthermore $\frac{\mathcal U(O_L)}{N}$ is a finite $A$-module.
  Let's set this time :
  $$a_j=\exp_{\widetilde{C}}(v_j)\in O_L[z].$$
  Let $M\in M_n(K[[z]])\cap M_n(\mathbb K_\infty)\subset M_n (\mathbb T(K_\infty)) $ be the  $n\times n$ matrix obtained by  expressing $v_j= \log_{\widetilde{C}}(a_j)$ in the same fixed $A$-basis of $O_L.$  Then there exists $\beta' \in \mathbb A$ such that :
  $$\det(M) =\beta' \det(M') = \beta' \alpha Z_{O_L}(1;z).$$
  But $Z_{O_L}(1;z)\in \mathbb T(K_\infty)^\times.$ Thus $\beta = \beta' \alpha \in A[z],$ and, Taelman's class formula implies that $\frac{\beta(1) }{[\frac{\mathcal U_{St}(O_L)}{N}]_A}\in \mathbb F_q^\times.$
  \end{proof}
  
  \begin{example} Let's finish this paragraph by a basic example : $L=K.$ We have:
$$\mathbb T(K_\infty)=A[z]\oplus \frac{1}{\theta}\mathbb F_q[z][[\frac{1}{\theta}]].$$
Observe that :
$$\Log_{\widetilde{C}}(1)\in 1+\frac{1}{\theta}\mathbb F_q[z][[\frac{1}{\theta}]].$$
Thus :
$$\mathbb T(K_\infty)= \Log_{\widetilde{C}}(1)A[z]\oplus  \frac{1}{\theta} \mathbb F_q[z][[\frac{1}{\theta}]].$$
We obtain :
$$\exp_{\widetilde{C}}(\mathbb T(K_\infty)) =U(A[z])\oplus   \frac{1}{\theta} \mathbb F_q[z][[\frac{1}{\theta}]],$$
and $U(A[z])$ is the sub-$A[z]$-module of $\widetilde{C}(A[z])$ generated by $1.$ Note that $U(A[z])$ is a free $A[z]$-module of rank one. Since $sgn(\Log_{\widetilde{C}}(1))=1,$ we get :
$$Z_A(1;z)=\Log_{\widetilde{C}}(1).$$
\end{example}


\section{The $P$-adic Carlitz module}

\subsection{$P$-adic $z$-deformation}\label{P-z-deformation}
The content of this paragraph is inspired by   the concept  of  formal Drinfeld modules (see for example \cite{ROS}).\par
Let $P$ be a prime of $A$ of degree $d.$  Let $z$ be an indeterminate over $K_P=\mathbb F_{q^d}((P)).$ We set : $\mathbb A_P= \mathbb F_{q^d}(z)[[P]],$ $\mathbb K_P= \mathbb F _{q^d} (z)((P)),$ $\mathbb T(K_P)= \mathbb F_{q^d}[z]((P)).$ We denote by $v_P$ the $P$-adic valuation  on $\mathbb K_P,$ i.e. :
$$v_P(\sum_{i\geq i_0} \alpha_i P^i)= i_0, \, i_0\in \mathbb Z, \alpha_i \in \mathbb F_{q^d}(z),\alpha_{i_0}\not =0.$$ 
Let $\tau : \mathbb  K_P \rightarrow \mathbb K_P$ be the continuous morphism of $\mathbb F_q(z)$-algebras such that $\forall x\in K_P, \tau(x)=x^q.$ 
Let  $\mathbb K_P\{\{\tau \}\}$ and $\mathbb K_P\{\tau\}$ be the non commutative power series and polynomials in $\tau$ with coefficients in $\mathbb K_P.$ 
If  $ a\in A,$ we set  :
$$\widetilde{C}_a=\sum_{i\geq 0} [a,i]z^i\tau ^i, [a,i]\in A, \, \forall  i>\deg_\theta a, [a,i]=0.$$
Recall that we have the Goss relations (\cite{GOS}, Proposition 3.3.10) :
$$\forall a\in A, [a,0]=a, \forall  i \geq 1,[a,i]=\frac{[a,i-1]^q-[a,i-1]}{\theta^{q^i}-\theta}.$$
 
\begin{lemma}\label{Lemma9} The map $\widetilde{C}$ extends uniquely into a continuous morphism of $\mathbb F_q(z)$-algebras $\widetilde{C}: \mathbb A_P \rightarrow \mathbb A_P\{\{\tau\}\}.$ In particular the Carlitz module extends uniquely into a continuous morphism of $\mathbb F_q$-algebras $C:  A_P \rightarrow A_P\{\{\tau\}\}.$
\end{lemma}
\begin{proof} Let $i\geq 0.$ Let $n\geq i,$ let  $a,b\in A$ such that $a\equiv b\pmod{P^n}.$ By the Goss relations, we get :
$$v_P([a,i]-[b,i])\geq n-i.$$
 The Lemma follows.
\end{proof}
 The extension of the Carlitz module  into a morphism from  $A_P$ to $A_P\{\{\tau\}\}$ in the above Lemma is called the $P$-adic Carlitz module, and, we  still denote  this object by $C$ in order to avoid heavy notation. 
Note that we have in $\mathbb K_P\{\{\tau\}\}$ :
$$\forall a\in \mathbb A_P, \exp_{\widetilde{C}} a = \widetilde{C}_a \exp_{\widetilde{C}},$$
$$\forall a\in \mathbb A_P, \Log_{\widetilde{C}} \widetilde{C}_a = a \Log_{\widetilde{C}}.$$
Furthermore, $\forall a\in A_P,$ we can write $\widetilde{C}_a= \sum_{i\geq 0} [a,i] z^i \tau ^i,$ $[a,i] \in A_P,$ and we still have the Goss relations :
$$[a,0]=a, \forall  i \geq 1,[a,i]=\frac{[a,i-1]^q-[a,i-1]}{\theta^{q^i}-\theta}.$$
 The same assertions remain true for the $P$-adic Carlitz module.
Let's observe that:
$$\forall i\geq 0, v_P (L_i)=[\frac{i}{d}],$$
where $[.]$ denotes the integral part, and :
$$v_P(D_i)= \frac{q ^i -q^{i-d[\frac{i}{d}]}}{q^d-1}.$$
Let $E/K_P$ be a finite extension and let $\mathbb E$ be the $P$-adic completion of $E(z).$  Let $v_P$ be the $P$-adic valuation on $E$ normalized such that $v_P(P)=1.$ We observe that $\forall x\in \mathbb E$ such that $v_P(x)>0,$ then $\Log_{\widetilde{C}}(x)$ converges in $\mathbb E.$ Furthermore :
$$\forall x\in \mathbb E, v_P(x) >0, \forall a\in \mathbb A_P,  \Log_{\widetilde{C}}(\widetilde{C}_a(x))=a\Log_{\widetilde{C}}(x),$$
$$\forall x\in \mathbb E, v_P(x)>\frac{1}{q^d-1}, v_P(\Log_{\widetilde{C}}(x))=v_P(x).$$
We observe that $\forall x\in \mathbb E$ is such that $v_P(x)>\frac{1}{q^d-1},$ then $\exp_{\widetilde{C}}(x)$ converges in $\mathbb E.$ Furthermore :

$$\forall x\in \mathbb E, v_P(x)>\frac{1}{q^d-1}, v_P(\exp_{\widetilde{C}}(x))=v_P(x),$$
$$\forall x\in \mathbb E, v_P(x)>\frac{1}{q^d-1}, \forall a\in \mathbb A_P,  \exp_{\widetilde{C}}(ax)=\widetilde{C}_a(\exp_{\widetilde{C}}(x)).$$
Finally, we also observe that :
$$\forall x\in \mathbb E, v_P(x)>\frac{1}{q^d-1}, \exp_{\widetilde{C}}(\Log_{\widetilde{C}}(x))= \Log_{\widetilde{C}}(\exp_{\widetilde{C}}(x))=x.$$
The same assertions remain valid for the Carlitz module.\par

\begin{lemma}\label{Lemma10} Let $x\in \mathbb E,v_P(x)>0.$ Then :
$$\Log_{\widetilde{C}}(x)=0\Leftrightarrow x=0.$$
 \end{lemma}
 
 \begin{proof} Let $\mathbb T(E)\subset \mathbb E$ be the  Tate algebra in the variable $z$ with coefficients in $E.$ We set :
 $$\mathbb T^{00}(E)=\{ x\in \mathbb T(E), v_P(x)>0\}.$$  Then $\Log_{\widetilde{C}}: \mathbb T^{00}(E)\rightarrow \mathbb T(E)$ is injective. Let $W$ be the sub-$\mathbb F_q(z)$-vector space of $\mathbb E$ generated by $\mathbb T^{00}(E).$  Then $W$ is dense in $\{ x\in \mathbb E, v_P(x)>0\},$ and $\Log_{\widetilde{C}}: W \rightarrow \mathbb E$ is injective. Now, let $x\in \mathbb E,$ $v_P(x)>0$ and $\Log_{\widetilde{C}}(x)=0.$ Let us write :
 $$x=w+ v, w\in W,  v\in \mathbb E, v_P(v)>\frac{1}{q^d-1}.$$
 We get :
 $$\Log_{\widetilde{C}}(v)=\Log_{\widetilde{C}}(-w).$$
 Thus :
 $$v=\exp_{\widetilde{C}}(\Log_{\widetilde{C}}(v))\in W.$$
 Thus $x\in W$ and therefore $x=0.$
 \end{proof}
 Let $f_E$ be the residual degree of $E/K_P.$ Let $O_{\mathbb E} =\{ x\in \mathbb E, v_P(x)\geq 0\}.$ Then, by the proof of Lemma \ref{Lemma5}, we have  :
 $$\widetilde{C}_{P^{f_E}-z^{df_E}}(O_{\mathbb E}) \subset \{ x\in \mathbb E, v_P(x)> 0\}.$$
 Now, observe that the map $\widetilde{C}_{P^{f_E}-z^{df_E}}: \mathbb T(E)\rightarrow \mathbb T(E)$ is injective, and therefore it is injective on $O_{\mathbb E}$. Thus, by  Lemma \ref{Lemma10}, we can conclude that :
 $$\widetilde{C}(O_{\mathbb E})_{\rm Tors}=\{0\},$$
 where  $\widetilde{C}(O_{\mathbb E})_{\rm Tors}$ denotes the torsion $\mathbb A$-sub-module of $\widetilde{C}(O_{\mathbb E})$ (note that $\widetilde{C}(O_{\mathbb E})$ is an $\mathbb A$-module and not an $\mathbb A_P$-module).
 Furthermore we define the $P$-adic Iwasawa logarithm attached to $\widetilde{C}$ as follows :
 $$\forall x\in O_{\mathbb E}, \Log_{\widetilde{C}}(x):= \frac{1}{P^{f_E}-z^{df_E}} \Log_{\widetilde{C}}(\widetilde{C}_{P^{f_E}-z^{df_E}}(x))\in \mathbb E.$$
 The map $\Log_{\widetilde{C}} : O_{\mathbb E}\rightarrow \mathbb E$ is a injective morphism of $\mathbb F_q(z)$-modules and :
 $$\forall a\in \mathbb A, \forall x\in O_{\mathbb E}, \Log_{\widetilde{C}}(\widetilde{C}_a(x))=
 a\Log_{\widetilde{C}}(x).$$
 In the case of the Carlitz module, if $O_E$ is the valuation ring of $E,$ we define the $P$-adic Iwasawa logarithm as follows :
 $$\forall x\in O_{E}, \Log_C(x):= \Log_{\widetilde{C}}(x)\mid _{z=1} \in E.$$
 Then :
 $$\forall a\in  A, \forall x\in O_{ E}, \Log_{C}(C_a(x))=
 a\Log_{C}(x),$$
 $$ {\rm Ker}(\Log_C: O_E\rightarrow O_E)=C(O_E)_{\rm Tors}.$$

\subsection{ A $P$-adic class formula} \label{P_adic class}
Let $L/K$ be a finite extension. Let :
$$L_P= L\otimes_K K_P,$$
$$\mathbb L_P =L\otimes _K \mathbb K_P,$$
$$\mathbb T(L_P)= L\otimes _K \mathbb T(K_P),$$
where we recall that $\mathbb T(K_P)$ is the Tate algebra in the variable $z$ with coefficients in $K_P.$ Recall that (Lemma \ref{Lemma3}, Lemma \ref{Lemma5}):
$$\widetilde{C}(\frac{\mathbb O_{L}}{P\mathbb O_{\ L} }) \simeq \prod_{\frak P\mid P} \frac{\mathbb A}{(n_{L/K}(\frak P)-z^{\deg_\theta(n_{L/K}(\frak P)}) \mathbb A},$$
$$ C(\frac{O_{ L}}{PO_{L} }) \simeq \prod_{\frak P\mid P} \frac{A}{(n_{L/K}(\frak P)-1)  A}.$$
Let :
$$O_{\mathbb L,P}= O_L \otimes _A \mathbb A_P,$$
$$O_{L,P}= O_L \otimes _A A_P.$$
We have  an injective morphism of $\mathbb F_q(z)$-vector spaces :
$$\Log_{\widetilde{C}} : O_{\mathbb L,P} \rightarrow \mathbb L_P,$$
 such that :
 $$\forall a\in \mathbb A, \forall x\in O_{\mathbb L,P}, \Log_{\widetilde{C}}(\widetilde{C}_a(x))= a\Log_{\widetilde{C}}(x).$$

We set :
$$\mathcal U(P\mathbb O_{ L})=\{ x\in \mathbb L_\infty \mid \exp_{\widetilde{C}}(x) \in P\mathbb O_{L}\}.$$

\begin{lemma}\label{Lemma11} There exists an isomorphism of $\mathbb A$-modules :
$$\frac{\mathcal U(\mathbb O_{ L})}{\mathcal U(P\mathbb O_{L})}\simeq \widetilde{C}(\frac{\mathbb O_{L}}{P\mathbb O_{L} }).$$
\end{lemma}
\begin{proof} Let $a\in A\setminus\{0\}.$ If we apply \cite{ATR}, Proposition 2, to the Drinfeld module $a^{-1} Ca,$   we get  :
$$\mathbb L_\infty = a\mathbb O_{ L} +\exp_{\widetilde{C}}(\mathbb L_\infty).$$
Therefore  we have  two short exacts sequences of $\mathbb A$-modules :
$$0\rightarrow \mathcal U(\mathbb O_{ L})\rightarrow\widetilde{C}( \mathbb O_{ L}) \rightarrow \frac{\mathbb L_\infty}{\exp_{\widetilde{C}}(\mathbb L_\infty)} \rightarrow 0,$$
$$0\rightarrow \mathcal U(a\mathbb O_{L})\rightarrow \widetilde{C}(a\mathbb O_{ L})\rightarrow \frac{\mathbb L_\infty}{\exp_{\widetilde{C}}(\mathbb L_\infty)} \rightarrow 0.$$
We therefore have an isomorphism of $\mathbb A$-modules :
$$\frac{\mathcal U(\mathbb O_{ L})}{\mathcal U(a\mathbb O_{ L})}\simeq \widetilde{C}(\frac{\mathbb O_{ L}}{a\mathbb O_{L} }).$$
\end{proof}
Let :
$$U(P\mathbb O_{L})=\exp_{\widetilde{C}}(\mathcal U(P\mathbb O_{ L}))\subset \widetilde{C}(P\mathbb O_{L}).$$
Recall that by the results of paragraph \ref{P-z-deformation}, $\Log_{\widetilde{C}}: \widetilde{C}(PO_{\mathbb L,P})\rightarrow PO_{\mathbb L, P}$ is an isomorphism of $\mathbb A$-modules. Let $\overline{U(P\mathbb O_{ L})}$ be the sub-$\mathbb A_P$-module of $\widetilde{C}(PO_{\mathbb L,P})$ generated by $U(P\mathbb O_{L}).$
 \begin{lemma}\label{Lemma12} The $\mathbb A_P$-module $\overline{U(P\mathbb O_{L})}$ is free of rank $[L:K].$
 \end{lemma} 
\begin{proof}  Let $n=[L:K].$ 
Let $\phi$ be the Drinfeld module $P^{-1}CP.$ Then, if we apply the results in \cite{ATR}, section 2.5, the $A$-module $\frac{1}{P}\{ x\in \mathbb T(L_\infty)\mid \exp_{\widetilde{C}}(x) \in PO_L[z]\}\mid_{z=1}$ is the $A$-module of "Stark units" of $\phi/O_L;$ in particular this module is an $A$-lattice in $L_\infty.$  Thus,  we can choose  $a_1, \ldots, a_n \in U(PO_L[z])
  \subset PO_L[z], $ such that $\sum _{i=1}^n \log_{\widetilde{C}}(a_i)\mid_{z=1} A$ is an $A$-lattice in $L_\infty.$
  Let $ M \in M_n (K[[z]])\cap M_n (\mathbb K_\infty)\subset M_n (\mathbb T(K_\infty))$ is the matrix obtained by  expressing $\log_{\widetilde{C}}(a_i)$ in a fixed $A$-basis of $O_L.$ 
  We have : $$\det(M)\mid_{z=1} \not =0.$$
 In particular :
 $$ \det(M) \not =0.$$
 This implies that $\Log_{\widetilde{C}}(\overline{U(P\mathbb O_{ L})})$ is a free $\mathbb A_P$-module of rank $n.$ The Lemma follows.
   \end{proof}
We get :
\begin{theorem}\label{Theorem2} Let $L/K$ be a finite extension and let $P$ be a prime of $A.$ Then:
$${\rm Fitt}_{\mathbb A_P}\frac{\widetilde{C}(PO_{\mathbb L,P})}{\overline{U(P\mathbb O_{L})}}= Z_{P, O_L}(1;z) \mathbb A_P.$$
\end{theorem}

 \begin{proof} Let $n=[L:K].$ Recall that $\log_{\widetilde{C}}: \widetilde{C}(PO_{\mathbb L,P}) \rightarrow PO_{\mathbb L,P}$ is an isomorphism of $\mathbb A_P$-modules, thus  :
 $${\rm Fitt}_{\mathbb A_P}\frac{\widetilde{C}(PO_{\mathbb L,P})}{\overline{U(P\mathbb O_{L})}}={\rm Fitt}_{\mathbb A_P}\frac{PO_{\mathbb L,P}}{\Log_{\widetilde{C}}(\overline{U(P\mathbb O_{ L})})}.$$  
  We have :
 $$R_{\mathcal U(P\mathbb O_{L})}=[\frac{\mathcal U(\mathbb O_{L})}{\mathcal U(P\mathbb O_{L})}]_{\mathbb A}R_{\mathcal U(\mathbb O_{L})}.$$
Thus, by Lemma \ref{Lemma11} :
$$R_{\mathcal U(P\mathbb O_{L})}=[\widetilde{C}(\frac{\mathbb O_{L}}{P\mathbb O_{L} })]_{\mathbb A}R_{\mathcal U(\mathbb O_{L})}.$$
Therefore, by the class formula for $\widetilde{C}/\mathbb O_L,$ we have an equality in $\mathbb T(K_\infty)$  :
 $$R_{\mathcal U(P\mathbb O_{L})}=[\widetilde{C}(\frac{\mathbb O_{ L}}{P\mathbb O_{ L} })]_{\mathbb A} Z_{O_L}(1;z).$$
 Now, we observe that, by the above equality and   by Lemma \ref{Lemma5},  $R_{\mathcal U(P\mathbb O_{L})}= P^n Z_{P, O_L}(1;z)$ can be viewed as an element of $\mathbb A_P.$
 Now,  we choose an $\mathbb A$-basis $u_1, \ldots, u_n \in \mathcal U(P\mathbb O_L).$ Then, by Lemma \ref{Lemma12}, $\exp_{\widetilde{C}}(u_i), i=1, \ldots, n$ is an $\mathbb A_P$-basis of $\overline{U(P\mathbb O_{L})}.$ Thus :
 $$\Log_{\widetilde{C}}(\overline{U(P\mathbb O_{ L})})=\oplus_{i=1}^n u_i \mathbb A_P.$$
 This implies that we have :
$${\rm Fitt}_{\mathbb A_P}\frac{O_{\mathbb L,P}}{\Log_{\widetilde{C}}(\overline{U(P\mathbb O_{ L})})}=R_{\mathcal U(P\mathbb O_{L})}\mathbb A_P= P^nZ_{P, O_L}(1;z)\mathbb A_P.$$
Now :
$${\rm Fitt}_{\mathbb A_P}\frac{O_{\mathbb L,P}}{\Log_{\widetilde{C}}(\overline{U(P\mathbb O_{ L})})}=P^n {\rm Fitt}_{\mathbb A_P}\frac{PO_{\mathbb L,P}}{\Log_{\widetilde{C}}(\overline{U(P\mathbb O_{L})})}= P^n {\rm Fitt}_{\mathbb A_P}\frac{\widetilde{C}(PO_{\mathbb L,P})}{\overline{U(P\mathbb O_{L})})}.$$
 \end{proof}
 

 \begin{theorem}\label{Theorem3} Let $L/K$ be a finite extension  of degree $n,$ and let $P$ be a prime of $A.$  Let $U(PO_L)= U(O_L)\cap PO_L,$ and let $\overline{U(PO_L)}$ be the sub $A_P$-module of $C(PO_{L,P})$ generated by  $U(PO_L).$ Then :
$${\rm Fitt}_{A_P}\frac{C(PO_{L,P})}{\overline{U(PO_L)}}=\frac{[U(O_L):U(PO_L)]_A}{[C(\frac{O_L}{PO_L})]_A}\frac{ \zeta_{P,O_L}(1) }{[H(O_L)]_A}A_P.$$
In particular, if $L/K$ is not "real" then :
$$\zeta_{P,O_L}(1)=0.$$
 \end{theorem}
 \begin{proof} let $n=[L:K].$ By the proof of Lemma \ref{Lemma12}, there exist $a_1, \ldots, a_n \in U(O_L[z])\cap PO_L[z]$ such that   $u_1\mid_{z=1}, \ldots , u_n\mid_{z=1}$ is a $K_\infty$-basis of $L_\infty,$ , where   $u_i=\log_{\widetilde{C}}(a_i) , i=1, \ldots ,n.$  Let $M$ be the  $n\times n$ matrix obtained by  expressing $u_i$ in a fixed $A$-basis of $O_L,$ By the refined class formula (Theorem \ref{Theorem1}), we have the following formal equality in $K[[z]]$: 
 $$\det(M) =\beta(z) Z_{O_L}(1;z),$$
 where  $\beta(z) \in A[z]\setminus (z-1)A[z],$ and :
 $$\beta(1) A ={\rm Fitt}_A\frac{\mathcal {U}_{St}(O_L)}{N},$$ 
 where $N=\oplus_{i=1}^n u_i\mid_{z=1} A.$ 
 But,  since $u_1, \ldots, u_n \in \mathbb T(L_P)= O_L\otimes_A \mathbb T(K_P),$  thus $M\in M_n(\mathbb T(K_P)),$ and we observe that the latter formal  equality  can be interpreted as an equality in $\mathbb T( K_P):$
 $$\det(M) = \beta(z) \frac{P^n}{[\widetilde{C}(\frac{\mathbb O_L}{P\mathbb O_L})]_{\mathbb A} }Z_{P,O_L}(1;z).$$
 Note that :
 $$[\widetilde{C}(\frac{\mathbb O_L}{P\mathbb O_L})]_{\mathbb A} \in \mathbb T(K_P)^\times.$$
 let $\bar N$ be the  sub-$A_P$ module of $C(PO_{L,P})$ generated by $a_1\mid_{z=1}, \ldots, a_n\mid_{z=1},$ then $\frac{\overline{U(PO_L)}}{\bar N}$ is a finite $A_P$-module.
 Thus, we have :
 $$\det(M)\mid_{z=1}\not =0 \Leftrightarrow {\rm rank}_{A_P} \overline{U(PO_L)}=n .$$
 In particular, if $L/K$ is not real, then ${\rm rank}_{A_P} \overline{U(PO_L)}<n .$  Thus, if $L/K$ is not real, we have $\zeta_{P,O_L}(1)= Z_{P,O_L}(1;z)\mid_{z=1}=0.$ Now, let's assume that $L/K$ is real. Then :
 $${\rm Fitt}_{A_P}\frac{\overline{U(PO_L)}}{\bar N}=({\rm Fitt}_{A}\frac {{\mathcal U(PO_L)}}{N}) A_P ,$$
 where $\mathcal U(PO_L)=\{ x\in L_\infty, \exp_C(x)\in PO_L\} .$
 Furthermore :
 $${\rm Fitt}_{A_P}\frac{C(PO_{L,P})}{{\bar N}}=\frac{\beta(1)}{[C(\frac{O_L}{PO_L})]_A} \zeta_{P,O_L}(1)A_P.$$
 Since $[\mathcal U(O_L): \mathcal U_{St}(O_L)]_A= [H(O_L)]_A,$ we get :
 $${\rm Fitt}_{A_P}\frac{C(PO_{L,P})}{\overline{U(PO_L)}}=\frac{[U(O_L):U(PO_L)]_A}{[C(\frac{O_L}{PO_L})]_A}\frac{ \zeta_{P,O_L}(1) }{[H(O_L)]_A}A_P.$$

 \end{proof}
 Now, let's assume that $L/K$ is real.  Then $U(O_L)$ is a free $A$-module of rank $n=[L:K].$ Let's select an  $A$-basis of $U(O_L):$ $a_1, \ldots ,a_n \in C(O_L).$  Let's fix an $A$-basis of $O_L : b_1, \ldots, b_n.$ Let $\sigma : L_P \rightarrow L_P$ be the $K_P$-linear map such that  $\sigma (b_i)= \Log_C(a_i),$ where $\Log_C$ is the Iwasawa $P$-adic logarithm.  Then $\det \sigma$ is well-defined modulo $\mathbb F_q^\times.$
  We set :
  $$R_{P,L}=\det \sigma\in K_P \pmod{\mathbb F_q^\times}.$$
 By the proof of Theorem \ref{Theorem3},  since $R_{P,L}$ is well-defined modulo $\mathbb F_q^\times,$ we can select an $A$-bais of $U(O_L)$ such that we have the following $P$-adic class formula  :
 \begin{theorem}\label{Theorem4} Let $L/K$ be a finite  real extension and let $P$ be a prime of $A.$  We have :
$$\zeta_{P,O_L}(1)=[H(O_L)]_A \frac{[C(\frac{O_L}{PO_L})]_A}{[\frac{O_L}{PO_L}]_A}  R_{P,L}.$$
\end{theorem}

\begin{example} We finish this paragraph by a basic example: $L=K.$  Recall that  we have the following equality in $\mathbb T(K_\infty):$
$$Z_A(1;z)=\Log_{\widetilde{C}}(1).$$
Thus, we get the following equality in $\mathbb T(K_P):$
$$Z_{P,A}(1;z) =(1-\frac{z^d}{P})\Log_{\widetilde{C}}(1),$$
where $d=\deg_\theta P,$ and  $\Log_{\widetilde{C}}(.)$  has to be understood as  the $P$-adic Iwasawa logarithm. We can evaluate at $z=1:$
$$\zeta_{P,A}(1)=(1-\frac{1}{P})\Log_{C}(1),$$ 
where again $\Log_C(.)$  is the $P$-adic Iwasawa logarithm.\par
\end{example}

\subsection{Leopoldt's  defect}
Let $L/K$ be a finite extension and $P$ a prime of $A.$ We define the $P$-adic Leopoldt's defect of $U(O_L)\subset C(O_L)$ as follows:
$$\delta_P(U(O_L)):= {\rm rank}_AU(O_L)-{\rm rank}_{A_P} \overline{U(PO_L)}\in \mathbb N.$$
For example, since $1\in U(A)$ is a torsion point  if and only if $q=2,$ we get :
$$\delta_P(U(A))=0.$$
 V. Bosser, L. Taelman and the author proved that if $L$ is the $P$th  cyclotomic function field then (as a consequence of V. Bosser's appendix in \cite{ANTAE}) :  $\delta_P(U(O_L))=0.$
Note that if $L/K$ is "real" we have :
$$\delta_P(U(O_L))=0\Leftrightarrow R_{P,L} \not =0\Leftrightarrow \zeta_{P,O_L}(1) \not =0.$$
 It would be interesting to examine Leopoldt's Conjecture in this context, i.e. do we always have $\delta_P(U(O_L))=0 ?$ We refer the interested reader  to a forthcoming paper of V. Bosser and the author where the Leopoldt's defect is studied in the more general context of Drinfeld modules.\par

\end{document}